\documentclass[12pt]{amsart}
\usepackage{amsmath,amssymb,amsbsy,amsfonts,amsthm,latexsym,
                        amsopn,amstext,amsxtra,euscript,amscd,mathrsfs,color,bm}
                        
\usepackage{float} 
\usepackage[english]{babel}
\usepackage{mathtools}
\usepackage{todonotes}
\usepackage{url}
\usepackage[colorlinks,linkcolor=blue,anchorcolor=blue,citecolor=blue,backref=page]{hyperref}

\usepackage[norefs,nocites]{refcheck}

\begin{document}
%
%\renewcommand*{\backref}[1]{}
%\renewcommand*{\backrefalt}[4]{%
%    \ifcase #1 (Not cited.)%
%    \or        (p.\,#2)%
%    \else      (pp.\,#2)%
%    \fi}     
    
%%%%%%%%%%%%%%%%%%%%%%%%%%%%%%%%%%%%%%%%%%%%%%%%%%%%%%%%%%%%%%%%%%%%%%
\newtheorem{theorem}{Theorem}
\newtheorem{lemma}[theorem]{Lemma}
\newtheorem{example}[theorem]{Example}
\newtheorem{algol}{Algorithm}
\newtheorem{corollary}[theorem]{Corollary}
\newtheorem{prop}[theorem]{Proposition}
\newtheorem{proposition}[theorem]{Proposition}
\newtheorem{problem}[theorem]{Problem}
\newtheorem{conj}[theorem]{Conjecture}
\newtheorem{cor}[theorem]{Corollary}

\newtheorem{definition}[theorem]{Definition}
\newtheorem{question}[theorem]{Question}
\newtheorem{remark}[theorem]{Remark}
\newtheorem*{acknowledgement}{Acknowledgements}

\theoremstyle{remark}

\numberwithin{equation}{section}
\numberwithin{theorem}{section}
\numberwithin{table}{section}
\numberwithin{figure}{section}

\allowdisplaybreaks

%%%%%%%%%%%%%%%%%%%%%%%%%%%%%%%%%%%%%%%%%%%%%%%%%%%%%%%%%%%%%%%%%%%%%%
\definecolor{olive}{rgb}{0.3, 0.4, .1}
\definecolor{dgreen}{rgb}{0.,0.5,0.}

\def\cc#1{\textcolor{red}{#1}} 

\definecolor{dgreen}{rgb}{0.,0.6,0.}
\def\tgreen#1{\begin{color}{dgreen}{\it{#1}}\end{color}}
\def\tblue#1{\begin{color}{blue}{\it{#1}}\end{color}}
\def\tred#1{\begin{color}{red}#1\end{color}}
\def\tmagenta#1{\begin{color}{magenta}{\it{#1}}\end{color}}
\def\tNavyBlue#1{\begin{color}{NavyBlue}{\it{#1}}\end{color}}
\def\tMaroon#1{\begin{color}{Maroon}{\it{#1}}\end{color}}

%\newcommand{\commA}[1]{\marginpar{%
%\begin{color}{blue}
%\vskip-\baselineskip %raise the marginpar a bit
%\raggedright\footnotesize
%\itshape\hrule \smallskip AO: #1\par\smallskip\hrule\end{color}}\ignorespaces}
%
%\newcommand{\commI}[1]{\marginpar{%
%\begin{color}{magenta}
%\vskip-\baselineskip %raise the marginpar a bit
%\raggedright\footnotesize
%\itshape\hrule \smallskip IS: #1\par\smallskip\hrule\end{color}}\ignorespaces}

%%%%%%%%%%%%%%%%%%%%%%%%%%%%%%%%%%%%%%%%%%%%%%%%%%%%%%%%%%%%%%%%%%%%%%

%%%%%%%%%%%%%%%%%%%%%%%%%%%%%%%%%%%%%%%%%%%%%%%%%%%%%%%%%%%%%%%%%%%%%%
%% %% Macros that are not used or that were defined twice
%% \def\xxx{\vskip5pt\hrule\vskip5pt}
%% \def\Cmt#1{\underline{{\sl Comments:}} {\it{#1}}}
%% \newcommand{\Modp}[1]{
%% \begin{color}{blue}
%%  #1\end{color}}
%% \def\bl#1{\begin{color}{blue}#1\end{color}} 
%% \def\red#1{\begin{color}{red}#1\end{color}} 
%\newcommand{\eqname}[1]{\tag{#1}}% Tag equation with name
%% JS: It is an abomination to redefine \(,\),\[,\]. They are STANDARD TeX macros
%% for delineating equations!!!
%% \def\({\left(}
%% \def\){\right)}
%% \def\[{\left[}
%% \def\]{\right]}
%% \def\gen#1{{\left\langle#1\right\rangle}}
%% \def\genp#1{{\left\langle#1\right\rangle}_p}
%% \def\genPs{{\left\langle P_1, \ldots, P_s\right\rangle}}
%% \def\genPsp{{\left\langle P_1, \ldots, P_s\right\rangle}_p}
%% \def\eq{\e_q}
%% \def\fh{{\mathfrak h}}
%% \def\fl#1{\left\lfloor#1\right\rfloor}
%% \def\rf#1{\left\lceil#1\right\rceil}
 \def\mand{\qquad\mbox{and}\qquad}
%% \def\jt{\tilde\jmath}
%% \def\ellmax{\ell_{\rm max}}
%% \def\rank#1{\mathrm{rank}#1} 
%% \def\m{{\rm m}}
%% \def\ch{\hat{h}}
%% \def\GL{{\rm GL}}
%% \def\Orb{\operatorname{Orb}}
%% \def\Per{\operatorname{Per}}
%% \def\Preper{\operatorname{Preper}}
%% \def\PGL{\operatorname{PGL}}
%% \def\tors{\operatorname{tors}}
%% \def\Gal{{\operatorname{Gal}}}
%% \newcommand{\Ch}{{\operatorname{Ch}}}
%% \newcommand{\Elim}{{\operatorname{Elim}}}
%% \newcommand{\proj}{{\operatorname{proj}}}
%% \newcommand{\h}{{\operatorname{\mathrm{h}}}}
%% \newcommand{\hh}{\mathrm{h}}
%% \newcommand{\bfalpha}{{\boldsymbol{\alpha}}}
%% \newcommand{\bfomega}{{\boldsymbol{\omega}}}
%%%%%%%%%%%%%%%%%%%%%%%%%%%%%%%%%%%%%%%%%%%%%%%%%%%%%%%%%%%%%%%%%%%%%%

%%%%%%%%%%%%%%%%%%%%%%%%%
% Alphabet calligraphic %
%%%%%%%%%%%%%%%%%%%%%%%%%
\def\cA{{\mathcal A}}
\def\cB{{\mathcal B}}
\def\cC{{\mathcal C}}
\def\cD{{\mathcal D}}
\def\cF{{\mathcal F}}
\def\cG{{\mathcal G}}
\def\cH{{\mathcal H}}
\def\cI{{\mathcal I}}
\def\cJ{{\mathcal J}}
\def\cK{{\mathcal K}}
\def\cL{{\mathcal L}}
\def\cM{{\mathcal M}}
\def\cN{{\mathcal N}}
\def\cO{{\mathcal O}}
\def\cP{{\mathcal P}}
\def\cQ{{\mathcal Q}}
\def\cR{{\mathcal R}}
\def\cS{{\mathcal S}}
\def\cT{{\mathcal T}}
\def\cU{{\mathcal U}}
\def\cV{{\mathcal V}}
\def\cW{{\mathcal W}}
\def\cX{{\mathcal X}}
\def\cY{{\mathcal Y}}
\def\cZ{{\mathcal Z}}

\def \sL {\mathscr{L}}
\def \sM {\mathscr{M}}

\def\C{\mathbb{C}}
\def\F{\mathbb{F}}
\def\K{\mathbb{K}}
\def\G{\mathbb{G}}
\def\Z{\mathbb{Z}}
\def\R{\mathbb{R}}
\def\Q{\mathbb{Q}}
\def\N{\mathbb{N}}
\def\L{\mathbb{L}}
\def\M{\textsf{M}}
\def\U{\mathbb{U}}
\def\P{\mathbb{P}}
\def\A{\mathbb{A}}
\def\fp{\mathfrak{p}}
\def\fq{\mathfrak{q}}
\def\n{\mathfrak{n}}
\def\fR{\mathfrak{R}}
\def\X{\mathcal{X}}
\def\x{\textrm{\bf x}}
\def\w{\textrm{\bf w}}
\def\ovQ{\overline{\Q}}
\def \Kab{\K^{\mathrm{ab}}}
\def \Qab{\Q^{\mathrm{ab}}}
\def \Qtr{\Q^{\mathrm{tr}}}
\def \Kc{\K^{\mathrm{c}}}
\def \Qc{\Q^{\mathrm{c}}}
\def\ZK{\Z_\K}
\def\ZKS{\Z_{\K,\cS}}
\def\ZKSf{\Z_{\K,\cS_f}}
\def\RSf{R_{\cS_{f}}}
\def\RTf{R_{\cT_{f}}}

\def\S{\mathcal{S}}
\def\vec#1{\mathbf{#1}}
\def\ov#1{{\overline{#1}}}
\def\Sp{{\operatorname{S}}}
\def\Gm{\G_{\textup{m}}}
\def\fA{{\mathfrak A}}
\def\fB{{\mathfrak B}}

\def \brho{\bm{\rho}}

\def\house#1{{%
    \setbox0=\hbox{$#1$}
    \vrule height \dimexpr\ht0+1.4pt width .5pt depth \dp0\relax
    \vrule height \dimexpr\ht0+1.4pt width \dimexpr\wd0+2pt depth \dimexpr-\ht0-1pt\relax
    \llap{$#1$\kern1pt}
    \vrule height \dimexpr\ht0+1.4pt width .5pt depth \dp0\relax}}

%%%%%%%%%%%%%%%%%%%%%%%%%%%%%%%%%%%%%%%%%%%%%%%%%%%%%%%%%%%%%%%%%%%%%%

%%%%%%%% Set Up Environment for Notation %%%%%%%%%%%%%%
% This is currently set to allow quite wide items to be defined
\newenvironment{notation}[0]{%
  \begin{list}%
    {}%
    {\setlength{\itemindent}{0pt}
     \setlength{\labelwidth}{1\parindent}
     \setlength{\labelsep}{\parindent}
     \setlength{\leftmargin}{2\parindent}
     \setlength{\itemsep}{0pt}
     }%
   }%
  {\end{list}}

%%%%%%%% Set Up Environment for Parts in Theorems %%%%%%%%%%%%%%
\newenvironment{parts}[0]{%
  \begin{list}{}%
    {\setlength{\itemindent}{0pt}
     \setlength{\labelwidth}{1.5\parindent}
     \setlength{\labelsep}{.5\parindent}
     \setlength{\leftmargin}{2\parindent}
     \setlength{\itemsep}{0pt}
     }%
   }%
  {\end{list}}
% Use \Part{(a)}, instead of \item[(a)], to ensure upright font
\newcommand{\Part}[1]{\item[\upshape#1]}

%%%%%%%% Set Up Macro for Cases %%%%%%%%%%%%%%
\def\Case#1#2{%
\smallskip\paragraph{\textbf{\boldmath Case #1: #2.}}\hfil\break\ignorespaces}

\def\Subcase#1#2{%
\smallskip\paragraph{\textit{\boldmath Subcase #1: #2.}}\hfil\break\ignorespaces}

%%%%%%%%%%%%%%%%%%
% Greek Alphabet %
%%%%%%%%%%%%%%%%%%
\renewcommand{\a}{\alpha}
\renewcommand{\b}{\beta}
\newcommand{\g}{\gamma}
\renewcommand{\d}{\delta}
\newcommand{\e}{\varepsilon}
\newcommand{\f}{\varphi}
\newcommand{\fhat}{\hat\varphi}
\newcommand{\bfphi}{{\boldsymbol{\f}}}
\renewcommand{\l}{\lambda}
\renewcommand{\k}{\kappa}
\newcommand{\lhat}{\hat\lambda}
\newcommand{\bfmu}{{\boldsymbol{\mu}}}
\renewcommand{\o}{\omega}
\renewcommand{\r}{\rho}
\newcommand{\rbar}{{\bar\rho}}
\newcommand{\s}{\sigma}
\newcommand{\sbar}{{\bar\sigma}}
\renewcommand{\t}{\tau}
\newcommand{\z}{\zeta}

%\newcommand{\D}{\Delta}
%\newcommand{\G}{\Gamma}
%\newcommand{\F}{\varphi}
%\renewcommand{\L}{\Lambda}

%%%%%%%%%%%%%%%%%%%%
% Fraktur Alphabet %
%%%%%%%%%%%%%%%%%%%%
\newcommand{\ga}{{\mathfrak{a}}}
\newcommand{\gb}{{\mathfrak{b}}}
\newcommand{\gn}{{\mathfrak{n}}}
\newcommand{\gp}{{\mathfrak{p}}}
\newcommand{\gP}{{\mathfrak{P}}}
\newcommand{\gq}{{\mathfrak{q}}}

%%%%%%%%%%%%%%%%%%%
% Barred Alphabet %
%%%%%%%%%%%%%%%%%%%
\newcommand{\Abar}{{\bar A}}
\newcommand{\Ebar}{{\bar E}}
\newcommand{\kbar}{{\bar k}}
\newcommand{\Kbar}{{\bar K}}
\newcommand{\Pbar}{{\bar P}}
\newcommand{\Sbar}{{\bar S}}
\newcommand{\Tbar}{{\bar T}}
\newcommand{\gbar}{{\bar\gamma}}
\newcommand{\lbar}{{\bar\lambda}}
\newcommand{\ybar}{{\bar y}}
\newcommand{\varphibar}{{\bar\f}}

%%%%%%%%%%%%%%%%%%%%%%%%%
% Calligraphic Alphabet %
%%%%%%%%%%%%%%%%%%%%%%%%%
\newcommand{\Acal}{{\mathcal A}}
\newcommand{\Bcal}{{\mathcal B}}
\newcommand{\Ccal}{{\mathcal C}}
\newcommand{\Dcal}{{\mathcal D}}
\newcommand{\Ecal}{{\mathcal E}}
\newcommand{\Fcal}{{\mathcal F}}
\newcommand{\Gcal}{{\mathcal G}}
\newcommand{\Hcal}{{\mathcal H}}
\newcommand{\Ical}{{\mathcal I}}
\newcommand{\Jcal}{{\mathcal J}}
\newcommand{\Kcal}{{\mathcal K}}
\newcommand{\Lcal}{{\mathcal L}}
\newcommand{\Mcal}{{\mathcal M}}
\newcommand{\Ncal}{{\mathcal N}}
\newcommand{\Ocal}{{\mathcal O}}
\newcommand{\Pcal}{{\mathcal P}}
\newcommand{\Qcal}{{\mathcal Q}}
\newcommand{\Rcal}{{\mathcal R}}
\newcommand{\Scal}{{\mathcal S}}
\newcommand{\Tcal}{{\mathcal T}}
\newcommand{\Ucal}{{\mathcal U}}
\newcommand{\Vcal}{{\mathcal V}}
\newcommand{\Wcal}{{\mathcal W}}
\newcommand{\Xcal}{{\mathcal X}}
\newcommand{\Ycal}{{\mathcal Y}}
\newcommand{\Zcal}{{\mathcal Z}}

%%%%%%%%%%%%%%%%%%%%%%%%%%%%
% Blackboard Bold Alphabet %
%%%%%%%%%%%%%%%%%%%%%%%%%%%%
\renewcommand{\AA}{\mathbb{A}}
\newcommand{\BB}{\mathbb{B}}
\newcommand{\CC}{\mathbb{C}}
\newcommand{\FF}{\mathbb{F}}
\newcommand{\GG}{\G}
\newcommand{\KK}{\mathbb{K}}
\newcommand{\NN}{\mathbb{N}}
\newcommand{\PP}{\mathbb{P}}
\newcommand{\QQ}{\mathbb{Q}}
\newcommand{\RR}{\mathbb{R}}
%\newcommand{\Z}{\mathbb{Z}}

%%%%%%%%%%%%%%%%%%%%%%%%%%
% Boldface Math Alphabet %
%%%%%%%%%%%%%%%%%%%%%%%%%%
\newcommand{\bfa}{{\boldsymbol a}}
\newcommand{\bfb}{{\boldsymbol b}}
\newcommand{\bfc}{{\boldsymbol c}}
\newcommand{\bfd}{{\boldsymbol d}}
\newcommand{\bfe}{{\boldsymbol e}}
\newcommand{\bff}{{\boldsymbol f}}
\newcommand{\bfg}{{\boldsymbol g}}
\newcommand{\bfi}{{\boldsymbol i}}
\newcommand{\bfj}{{\boldsymbol j}}
\newcommand{\bfk}{{\boldsymbol k}}
\newcommand{\bfm}{{\boldsymbol m}}
\newcommand{\bfp}{{\boldsymbol p}}
\newcommand{\bfr}{{\boldsymbol r}}
\newcommand{\bfs}{{\boldsymbol s}}
\newcommand{\bft}{{\boldsymbol t}}
\newcommand{\bfu}{{\boldsymbol u}}
\newcommand{\bfv}{{\boldsymbol v}}
\newcommand{\bfw}{{\boldsymbol w}}
\newcommand{\bfx}{{\boldsymbol x}}
\newcommand{\bfy}{{\boldsymbol y}}
\newcommand{\bfz}{{\boldsymbol z}}
\newcommand{\bfA}{{\boldsymbol A}}
\newcommand{\bfF}{{\boldsymbol F}}
\newcommand{\bfB}{{\boldsymbol B}}
\newcommand{\bfD}{{\boldsymbol D}}
\newcommand{\bfG}{{\boldsymbol G}}
\newcommand{\bfI}{{\boldsymbol I}}
\newcommand{\bfM}{{\boldsymbol M}}
\newcommand{\bfP}{{\boldsymbol P}}
\newcommand{\bfX}{{\boldsymbol X}}
\newcommand{\bfY}{{\boldsymbol Y}}
\newcommand{\bfzero}{{\boldsymbol{0}}}
\newcommand{\bfone}{{\boldsymbol{1}}}

%%%%%%%%%%%%%%%%%%%%%%%%%%%%%%
% Miscellaneous New Commands %
%%%%%%%%%%%%%%%%%%%%%%%%%%%%%%
\newcommand{\aff}{{\textup{aff}}}
\newcommand{\Aut}{\operatorname{Aut}}
\newcommand{\Berk}{{\textup{Berk}}}
\newcommand{\Birat}{\operatorname{Birat}}
\newcommand{\characteristic}{\operatorname{char}}
\newcommand{\codim}{\operatorname{codim}}
\newcommand{\Crit}{\operatorname{Crit}}
\newcommand{\critwt}{\operatorname{critwt}} % valency of a portrait
\newcommand{\Cycle}{\operatorname{Cycles}}
\newcommand{\diag}{\operatorname{diag}}
\newcommand{\Disc}{\operatorname{Disc}}
\newcommand{\Div}{\operatorname{Div}}
\newcommand{\Dom}{\operatorname{Dom}}
\newcommand{\End}{\operatorname{End}}
\newcommand{\ExtOrbit}{\mathcal{EO}} %% Extended orbit
\newcommand{\Fbar}{{\bar{F}}}
\newcommand{\Fix}{\operatorname{Fix}}
\newcommand{\FOD}{\operatorname{FOD}}
\newcommand{\FOM}{\operatorname{FOM}}
\newcommand{\Gal}{\operatorname{Gal}}
\newcommand{\genus}{\operatorname{genus}}
\newcommand{\GITQuot}{/\!/}
\newcommand{\GL}{\operatorname{GL}}
\newcommand{\GR}{\operatorname{\mathcal{G\!R}}}
\newcommand{\Hom}{\operatorname{Hom}}
\newcommand{\Index}{\operatorname{Index}}
\newcommand{\Image}{\operatorname{Image}}
\newcommand{\Isom}{\operatorname{Isom}}
\newcommand{\hhat}{{\hat h}}
\newcommand{\Ker}{{\operatorname{ker}}}
\newcommand{\Ksep}{K^{\textup{sep}}}  %% separable closure of K
\newcommand{\lcm}{{\operatorname{lcm}}}
\newcommand{\LCM}{{\operatorname{LCM}}}
\newcommand{\Lift}{\operatorname{Lift}}
\newcommand{\limstar}{\lim\nolimits^*}
\newcommand{\limstarn}{\lim_{\hidewidth n\to\infty\hidewidth}{\!}^*{\,}}
\newcommand{\llog}{\log\log}
\newcommand{\logplus}{\log^{\scriptscriptstyle+}}
\newcommand{\Mat}{\operatorname{Mat}}
\newcommand{\maxplus}{\operatornamewithlimits{\textup{max}^{\scriptscriptstyle+}}}
\newcommand{\MOD}[1]{~(\textup{mod}~#1)}
\newcommand{\Mor}{\operatorname{Mor}}
\newcommand{\Moduli}{\mathcal{M}}
\newcommand{\Norm}{{\operatorname{\mathsf{N}}}}
\newcommand{\notdivide}{\nmid}
\newcommand{\normalsubgroup}{\triangleleft}
\newcommand{\NS}{\operatorname{NS}}
\newcommand{\onto}{\twoheadrightarrow}
\newcommand{\ord}{\operatorname{ord}}
\newcommand{\Orbit}{\mathcal{O}}
\newcommand{\Per}{\operatorname{Per}}
\newcommand{\Perp}{\operatorname{Perp}}
\newcommand{\PrePer}{\operatorname{PrePer}}
\newcommand{\PGL}{\operatorname{PGL}}
\newcommand{\Pic}{\operatorname{Pic}}
\newcommand{\Prob}{\operatorname{Prob}}
\newcommand{\Proj}{\operatorname{Proj}}
\newcommand{\Qbar}{{\bar{\QQ}}}
\newcommand{\rank}{\operatorname{rank}}
\newcommand{\Rat}{\operatorname{Rat}}
\newcommand{\Res}{{\operatorname{Res}}}
\newcommand{\Resultant}{\operatorname{Res}}
\renewcommand{\setminus}{\smallsetminus}
\newcommand{\sgn}{\operatorname{sgn}}
\newcommand{\SL}{\operatorname{SL}}
\newcommand{\Span}{\operatorname{Span}}
\newcommand{\Spec}{\operatorname{Spec}}
\renewcommand{\ss}{{\textup{ss}}}
\newcommand{\stab}{{\textup{stab}}}
\newcommand{\Stab}{\operatorname{Stab}}
\newcommand{\Support}{\operatorname{Supp}}
\newcommand{\Sym}{\operatorname{Sym}}  %% Symmetric group
\newcommand{\tors}{{\textup{tors}}}
\newcommand{\Trace}{\operatorname{Trace}}
\newcommand{\trianglebin}{\mathbin{\triangle}} % symmetric set difference
\newcommand{\tr}{{\textup{tr}}} % for K/k trace
\newcommand{\UHP}{{\mathfrak{h}}}    % Upper half plane
\newcommand{\Wander}{\operatorname{Wander}}
\newcommand{\<}{\langle}
\renewcommand{\>}{\rangle}

\newcommand{\pmodintext}[1]{~\textup{(mod}~#1\textup{)}}
\newcommand{\ds}{\displaystyle}
\newcommand{\longhookrightarrow}{\lhook\joinrel\longrightarrow}
\newcommand{\longonto}{\relbar\joinrel\twoheadrightarrow}
\newcommand{\SmallMatrix}[1]{%
  \left(\begin{smallmatrix} #1 \end{smallmatrix}\right)}
  
  \def\({\left(}
\def\){\right)}
\def\fl#1{\left\lfloor#1\right\rfloor}
\def\rf#1{\left\lceil#1\right\rceil}
%\def\[{\left[}
%\def\]{\right]}
%\def\<{\langle}
%\def\>{\rangle}

  %%%%%%%%%%%%%%%%%%%%%%%%%%%%%%%%%%

\title[  elements of large order of elliptic curves and functions
 ]
{On elements of large order of elliptic curves and multiplicative dependent images of rational functions over finite fields} 

\author[B.\ Kerr]{Bryce Kerr}
\address{Department of Mathematics and Statistics,
University of Turku,  FI-20014, Finland}
\email{bryce.kerr@utu.fi}

\author[J. Mello] {Jorge Mello}
\address{School of Mathematics and Statistics, University of New South Wales, Sydney NSW 2052, Australia}
\email{j.mello@unsw.edu.au}

\author[I. E. Shparlinski] {Igor E. Shparlinski}
\address{School of Mathematics and Statistics, University of New South Wales, Sydney NSW 2052, Australia}
\email{igor.shparlinski@unsw.edu.au}

\subjclass[2010]{11G05, 11G07}
\keywords{elliptic curve, order of a point} 

\begin{abstract} 
Let $E_1$ and $E_2$ be elliptic curves in Legendre form with integer parameters. We show there exists a constant $C$ such that for almost all primes, for all but at most $C$ pairs of points on the reduction of $E_1 \times E_2$ modulo $p$ having equal $x$ coordinate, at least one among $P_1$ and $P_2$ has a large group order. We also show similar abundance over finite fields of elements whose images under the reduction modulo $p$ of a 
finite set of rational functions have large multiplicative orders.
\end{abstract}

\maketitle

%%%%%%%%%%%%%%%%%%%%%%%%%%%%%%%%%%
%%%%%%%%%%%%%%%%%%%%%%%%%%%%%%%%%%%%
\section{Introduction}

\subsection{Description of our results}
In this paper we consider some variants in positive characteristic of 
characteristic zero results which are generically called 
{\it  unlikely intersections\/}. In particular, we give new estimates for 
 
\begin{itemize}
\item Lower bounds on orders of  points on elliptic curves over finite fields, see Section~\ref{sec:torsEC};
\item Lower bounds on multiplicative orders of  reductions of points on some varieties over $\C$,
see Section~\ref{sec:ord_var};
\end{itemize}
 These  results complement those of~\cite{BCMOS, CKSZ, Mello, Shp1, Shp2} and may be considered as nonzero characteristic variants of results of De Marco, Krieger and Ye~\cite{DKY,DKX} concerning torsion points on elliptic curves,  and also a nonzero characteristic analogue  of a result of Bombieri, Masser and Zannier~\cite{BMZ} concerning multiplicative relations between rational functions.

%In particular, we consider counting points of large order on varieties over prime fields and to the sum-product problem 
% in prime fields for sets with very small sumset.

\subsection{General notation}
\label{sec:notation}
Throughout this work $\N = \{1, 2, \ldots\}$ is the set of positive integers.
We also write $\Z_{> a}$ for the set of $n \in \Z$ with $n > a$ and similarly 
for  $\Z_{< a}$.  

For a field $K$ we use $\ov K$ to denote the algebraic closure of $K$. 

For a prime $p$, we use $\F_p$ to denote the finite field of $p$ elements. 

The letters $k$, $\ell$  $m$ and $n$ (with or without subscripts) are always used denote positive integers;
the letter~$p$  
(with or without subscripts) is always used to denote a prime.

As usual, for given quantities  $U$ and $V$, the notations $U\ll V$, $V\gg  U$ and
$U=O(V)$ are all equivalent to the statement that the inequality
$|U|\le c V$ holds with some absolute constant $c>0$.

Throughout the paper, any implied constants in symbols $O$, $\ll$
and $\gg$ may depend on the parameters of globally defined objects, such as coefficients of 
Weierstrass equations of elliptic curves or  coefficients and degrees of polynomials defined over 
$\overline{\mathbb{Q}}$, 
and are \emph{absolute} unless specified otherwise.
%To indicate dependence on the vector $\brho$ of these parameters, we sometimes, where it is not obvious, put then 
%the subscripts: for example $O_{\brho}$, or  $\gg_{\brho}$. 

The following notion of multiplicative dependence plays an important role in 
our argument. 

As usual, we say that the points $x_1,\ldots, x_n \in \overline{\mathbb{Q}}$ are multiplicatively dependent if  there exist $k_1,\ldots ,k_n \in \mathbb{Z}$ not all zero such that $x_1^{k_1}\ldots x_n^{k_n}=1$. 
If the points $x_1,\ldots,x_n\in \overline{\mathbb{Q}}$ are not multiplicatively dependent then we say they are multiplicatively independent. 

\begin{definition}
\label{def:ord} 
We define the \textit{multiplicative order} of  a  multiplicatively dependent tuple 
$(x_1,\ldots ,x_n)\in \ov \Q$ as 
%$$
%\ord (x_1,\ldots ,x_n)= \min_{(k_1,\ldots ,k_n) \in \mathbb{Z}^n\setminus \{\vec{0}\}} \max_{0 \leq i \leq n} \{|k_i| : ~x_1^{k_1}\ldots x_n^{k_n}=1\}.
%$$
\begin{align*}
\ord (x_1,\ldots ,x_n)= \min\{ \max_{1 \leq i \leq n} |k_i| :~(k_1,\ldots ,k_n)& \in \mathbb{Z}^n\setminus \{\vec{0}\},\\
&\quad  x_1^{k_1}\ldots x_n^{k_n}=1\}.
\end{align*} 
\end{definition}

We use $|\cS|$ to denote the cardinality of a finite set $\cS$. 

Finally, for a subset $\mathcal{P}$ of primes, its \textit{natural density} is defined as the real number
$$
\lim_{Q \to \infty} \frac{1}{\pi(Q)} \left|\left\{ p \in \mathcal{P} :~p \leq Q \right\}\right| ,
$$
whenever this limit exists, where, as usual 
$\pi(Q) = \left|\{ p:~p \leq Q \}\right|$.  

We say that a certain statement hold for {\it almost all primes\/} if it holds for a set of primes of natural density $1$.

\section{Main results}

\subsection{Torsion of points modular reductions of elliptic curves}
\label{sec:torsEC}

Given a point $P$ in the group of points $E(\ov K)$ on  an elliptic curve $E$ defined over  a field $K$, we denote by $\ord P$ 
the order of $P$ in the group of points on $E$ over the algebraic closure of $K$, see~\cite{Silv} for a background on elliptic curves. 

We also recall that points of finite order are called {\it torsion points\/}. 

%\commI{I think 2.1-2.2 must go as any EC has $O(t^2)$ of order $t$, so T2.1 holds for *any* p} 
%\commB{Removed 2.1--2.2}
%\begin{theorem}\label{th2.1}
% Let $E$ be an elliptic curve defined over $\mathbb{Q}$ and given by a Weierstrass equation $Y^2=X^3+aX+b$ with $a,b \in \mathbb{\mathbb{Z}}$. Let also $L \geq 1$.
%Then there exists an integer $\mathfrak{A}_L \geq 1$, multiple of the conductor of $E$, satisfying
%$$
%\log \mathfrak{A}_L \ll L^{24},
%$$
% with the implied constant depending only on $a$ and $b$, 
% such that, for every prime $p$ not dividing $\mathfrak{A}_L$, for all but at most $O(L^3)$  points $P\in E(\overline{\mathbb{F}}_p)$ of the reduction of $E$ modulo $p$, we have
% \begin{center}
 % $\ord P >L$.
 %\end{center}
%\end{theorem}

%\begin{corollary}
%\label{cor2.2}
% Under the conditions of Theorem \ref{th2.1}, for any $\varepsilon, \delta >0$, for any  sufficiently large $Q \geq 1$, for all but $o(Q/ \log Q)$ primes $p \leq Q$, we have
% \begin{equation}
% \ord P \geq \varepsilon p^{1/(24+ \delta)},
% \end{equation} for all but at most $O(\varepsilon^3Q^{1/8})$ points $P$ in $E(\overline{\mathbb{F}}_p)$.
%\end{corollary}

%\begin{corollary}
%\label{cor2.2}\commI{Reformulated} 
% Under the conditions of Theorem \ref{th2.1}, for any $ \delta >0$, for almost primes $p \leq Q$, we have
% \begin{equation}
 %\ord P \geq  p^{1/24}(\log p)^{1/24 -\delta} ,
% \end{equation} for all but at most $O(p^{1/8}(\log p)^{1/8 -3\delta} )$ points $P$ in $E(\overline{\mathbb{F}}_p)$.
%\end{corollary}
Our first result may be considered a nonzero characteristic variant of a theorem of De Marco, Krieger and Ye~\cite{DOSS} concerning torsion points on elliptic curves.

\begin{theorem}\label{th2.3}
 There is an absolute constant $C_0$ such that for any 
fixed  elliptic curves $E_1$ and $E_2$  in Legendre 
 form 
 $$E_1:Y^2=X(X-1)(X-t_1) \mand E_2:Y^2= X(X-1)(X-t_2)
 $$ 
 with distinct $t_1, t_2  \in \mathbb{Z} \setminus \{ 0,1 \}$,   for almost all  primes $p$, 
 for all but at most $C_0$ points $(P_1,P_2)$ in $E_1(\overline{\mathbb{F}}_p)\times E_2(\overline{\mathbb{F}}_p)$ with $x(P_1)=x(P_2)$, we have 
$$
 \max\{ \ord P_1,  \ord P_2\} \geq p^{1/6 + o(1)}.
$$
\end{theorem}

\subsection{Multiplicative orders of points on   modular reductions of  varieties} 
 \label{sec:ord_var} 

%Given a tuple of elements $(x_1,\ldots,x_n)$ in some field we define the \textit{multiplicative order} of $(x_1,\ldots,x_n)$ to be
%$$
%\ord (x_1,\ldots,x_n)= \min_{(k_1,\ldots,k_n) \in \mathbb{Z}^n\setminus \{\vec{0}\}}\,  \max_{0 \leq i \leq n} \{|k_i| :~ x_1^{k_1}\ldots x_n^{k_n}=1\}.
%$$ 

We say that nonzero rational functions  $f_1,\ldots,f_n \in \mathbb{Q}(X)$   are {\it multiplicatively independent \/} 
if there is no nontrivial product with $f_1^{\ell_1}(X)\ldots f_n^{\ell_n}(X)=1$.
We also recall Definition~\ref{def:ord}. 

%%\begin{theorem}\label{th2.9}
%%There is a constant constant $C_0$ that depends only on $d$ such that for any 
%%$$
%%\cF = \{ f_1,\ldots,f_n\}
%%$$   family of nonzero multiplicatively independent modulo constants rational functions
%%over $\Q$ of degrees at most $d$, and  for any function $\varepsilon(z)$ with $\lim_{z  \to \infty}\varepsilon(z) =0 $,  for almost all  primes $p$, for all but at most $C_0$ points $x \in\ov \F_p$ satisfying $$f_1^{k_1}(x)\ldots f_n^{k_n}(x)=f_1^{\ell_1}(x)\ldots f_n^{\ell_n}(x)=1 \pmod  p$$ for some linearly independent  integer vectors $(k_1,\ldots ,k_n), (\ell_1,\ldots ,\ell_n)$, we have 
%%$$
%%    \ord (f_1(x),\ldots ,f_n(x))\geq \varepsilon(p) p^{1/(2n+2)}. 
%%$$
%%\end{theorem}

%\begin{theorem}\label{th2.9}
%For any 
%$$
%\cF = \{ f_1,\ldots,f_n\}
%$$

\begin{theorem}\label{th2.9}
 For any  multiplicatively independent  rational functions $f_1,\ldots,f_n \in \Q(X)$
 there is an effectively computable constant constant $C_0$ that depends only on $f_1,\ldots,f_n$ such that for any function $\varepsilon(z)$ with $\lim_{z  \to \infty}\varepsilon(z) =0 $,  for almost all  primes $p$, for all but at most $C_0$ points $x \in\ov \F_p$ satisfying $$f_1^{k_1}(x)\ldots f_n^{k_n}(x)=f_1^{\ell_1}(x)\ldots f_n^{\ell_n}(x)=1  %%\pmod  p
 $$ 
 for some linearly independent  integer vectors $(k_1,\ldots ,k_n), (\ell_1,\ldots ,\ell_n)$, we have 
$$
    \ord (f_1(x),\ldots ,f_n(x))\geq \varepsilon(p) p^{1/(2n+2)}. 
$$
\end{theorem} 
 
We remark that Theorem~\ref{th2.9} complements some recent results of Barroero, Capuano,   M{\'e}rai,   Ostafe and Sha~\cite{BCMOS} and is based on similar technical tools.

\section{Preliminaries}

\subsection{Tools from Diophantine geometry}
\label{sec:diopantine}
For a polynomial $G$ with integer coefficients, its \textit{height}, denoted by $h(G)$, is defined as the logarithm of the maximum of the absolute values of the coefficients of $G$.

We recall the following well-known estimate, see, for example,~\cite[Lemma~1.2~(1.b) and~(1.d)]{KPS}.

\begin{lemma}\label{lem3.1}
 Let $G_i \in \mathbb{Z}[T_1,\ldots ,T_n],\  i=1,\ldots ,s$. Then
\begin{align*}
 \sum_{i=1}^sh(G_i) & - 2 \log (n+1) \sum_{i=1}^s \deg G_i \\
 & \le 
h\(\prod_{i=1}^sG_i\) \leq   \sum_{i=1}^sh(G_i) + \log (n+1) \sum_{i=1}^s \deg G_i.
\end{align*}
 \end{lemma}
 
We  also use an estimate on the height of sums of polynomials which is an easy consequence of the definition of height.
\begin{lemma}\label{lem3.11}
 Let $G_i \in \mathbb{Z}[T_1,\ldots ,T_n], i=1,\ldots ,s$. Then
$$
h\(\sum_{i=1}^sG_i\) \leq \  \max_{1\le i \le s}h(G_i)+\log{s}.
$$
 \end{lemma}

We also need a resultant bound, which follows from Hadamard's inequality, see for example,~\cite[Theorem~6.23]{vzGG}. 
%%~\cite[Theorem 7]{BL}.
\begin{lemma}\label{lemres}
Let 
$$A(X)=\sum_{i=1}^ m a_iX^i \mand B(X)=\sum_{j=1}^n b_jX^j
$$ 
be two polynomials in $\mathbb{C}[X]$ of respective degrees $m$ and $n$.
Then their  resultant $\Res(A,B)$ is bounded by
$$
|\Res(A,B)| \leq  \left(\sum_{i=1}^ m |a_i| ^2\right)^{n/2}  \left(\sum_{j=1}^n |b_j| ^2\right)^{m/2}. 
$$
\end{lemma}

\subsection{Tools from unlikely intersections}

For an elliptic curve $E$ over a field $K$ we use $E^{\text{tors}}$ to denote the set of all torsion points on $E\(\ov K\)$.  

By fixing coordinates on $\mathbb{P}^1$, we consider the Legendre family of elliptic curves
$$
E_t: Y^2=X(X-1)(X-t)
$$ with $t \in \mathbb{C}/\{0,1 \}$ and the standard projection $\pi(x,y)=x$ on $E_t$.

By a result of De Marco, Krieger and Ye~\cite[Theorem~1.4]{DKX}, we have:

\begin{lemma}\label{lem3.2}
There exist an %% uniform 
absolute constant $B$ such that 
$$
\left|\pi (E_{t_1}^{\mathrm{tors}}) \bigcap \pi (E_{t_2}^{\mathrm{tors}})\right| \leq B,
$$ for all $t_1 \neq t_2$ in $\mathbb{C} \setminus \{0,1 \}$. 
\end{lemma}

The next result of  Maurin~\cite[Th{\'e}or{\`e}me~1.2]{Mau},   which improves and makes effective the previous result of Bombieri,  Masser and Zannier~\cite[Theorem~2]{BMZ} (see also~\cite{CMPZ}), 
concerns intersections between a curve and subgroups of the $n$-dimensional torus with co-dimension at least $2$.
As usual, we use $\G_m= \Q^*$ to denote the multiplicative group of $\Q$.  This naturally transfers to a group 
structure on  $\G_m^n$. Following the terminology of~\cite{BMZ}, a connected algebraic subgroup $\cH$ of $\G_m^n$
 is called   a {\it torus\/}. If a torus $\cH \ne   \G_m^n$  call it a  {\it proper  subtorus\/}. Finally as set $\gamma\cH$,  with $\gamma \in \G_m$
 is called a translate of $\cH$.   Then by a special case of~\cite[Th{\'e}or{\`e}me~1.2]{Mau}, 
 we have the following.  
 
%\begin{lemma}\label{lem3.5}
%Let $\cC$ be a closed absolutely irreducible curve in $\G_m^n$, $n \geq 2$, defined over $\overline{\mathbb{Q}}$ and not contained in a translate of a proper subtorus of $\G_m^n$. Then the points of $\cC \cap \cH \subseteq \G_m^n$ for $\cH$ ranging over all algebraic subgroups of $\G_m^n$ with $\dim \cH \leq n-2$ form a finite set
%of size effectuvely 
%\end{lemma} 
%
%
%
%Lemma~\ref{lem3.5}  immediately implies.
%\commB{'not all zeros' $\rightarrow$ `not all zero'}
\begin{lemma}\label{lem:multidep}
Let $ f_1,\ldots,f_n \in \mathbb{Q}(X)$ be rational functions that are multiplicatively independent.  
Then the set of $\alpha \in \bar{\mathbb{Q}}$ such that 
$$f_1(\alpha)^{a_1}\ldots f _n(\alpha)^{a_n}=f_1(\alpha)^{b_1}\ldots f_n(\alpha)^{b_n}=1,
$$ 
for some  linearly independent vectors $(a_1,\ldots ,a_n), (b_1,\ldots ,b_n)\in \Z^n$ is finite 
of cardinality bounded by an effective constant depending only on $f_1,\ldots,f_n$. 
\end{lemma}

\subsection{Background on division polynomials}
\label{sec:division}
Here we give some preliminary estimates for division polynomials for elliptic curves in Legendre form. The results contained in this section are due to Ho~\cite[Chapter~4]{Ho} and are obtained in his master's thesis. Since this thesis may be difficult to access, we reproduce some details.

Let $E_{\lambda}$ be an elliptic curve given in Legendre form
$$
E_\lambda: Y^2=X(X-1)(X-\lambda).
$$
The division polynomials $\psi_k$ are defined recursively  by 
\begin{align*}
\psi_{2k+1}&=\psi_{k+2}\psi_k^3-\psi_{k-1}\psi_{k+1}^3, \\
\psi_{2k}&=\frac{1}{2Y}\psi_k(\psi_{k+2}\psi_{k-1}^2-\psi_{k-2}\psi_{k+1}^2),
\end{align*}
with initial values 
\begin{align*}
\psi_0&=0 \\
\psi_1&=1 \\
\psi_2&=2Y \\
\psi_3&=3X^4-4(1+\lambda)X^3+6\lambda X^2-\lambda^2 \\
\psi_4&=2Y(2X^6-4(1+\lambda)X^5+10\lambda X^4\\
& \qquad \qquad \qquad \qquad \qquad -10\lambda^2X^2+4\lambda^2(1+\lambda)X-2\lambda^3).
\end{align*}

We note that similar (but slightly different) polynomials have also been introduced by Stoll~\cite[Section~3]{Stoll}.
However this definition suits our purpose better.

Define 
\begin{align*}
\phi_n&=X\psi_n^2-\psi_{n+1}\psi_{n-1} \\
4Y\omega_n&=\psi_{n-1}^2\psi_{n+2}-\psi_{n-2}\psi_{n+1}^2.
\end{align*}

Arguing as in~\cite[Excercise~3.7]{Silv}, we have:

\begin{lemma}
\label{lem:torsion1}
For any $P\in E_{\lambda}$ and $n\ge 1$ we have 
$$
[n]P=\left(\frac{\phi_n(P)}{\psi_n(P)^2},\frac{\omega_n(P)}{\psi_n(P)^2}\right).
$$
\end{lemma}

We need some basic  properties of these polynomials which are contained 
in the  Master's thesis of Ho~\cite[Chapter~4]{Ho}.    For the sake of completeness, 
we prove these results in Appendix~\ref{app:div poly}. Our next result 
collects together the statements of Lemmas~\ref{lem:polyz}, \ref{lem:polydeg} and~\ref{lem:polyht}.

\begin{lemma} \label{lem:div poly}
The rational functions $\psi_n$ are polynomials in the ring $ \Z[\lambda, X, Y]$ of 
degree $\deg \psi_n \le n^{2+o(1)}$ and of height $h(\psi_n) \le n^{2+o(1)}$.
\end{lemma}

 \subsection{Proof of Theorem~\ref{th2.3}}
 By Lemma~\ref{lem3.2} there is an absolute constant $C$ such that the components of any pair 
$$(P_1,P_2)\in E_{1}(\bar{\mathbb{Q}})^{\text{tors}}\times E_{2}(\bar{\mathbb{Q}})^{\text{tors}}
$$ 
with $x(P_1)=x(P_2)$ are torsion points of order at most $C$.  
Let us fix some $\varepsilon > 0$. 
We assume that $z$ is large enough and fixed so that 
\begin{equation}
\label{eq:Ldef1}
L =z^{1/6- \varepsilon} >C.
\end{equation} 
Consider the curve
$$E_1: Y=X(X-1)(X-t_1).$$
With notation as in Section~\ref{sec:division},  let $\psi_n$ be the division polynomials for the curve $E_1$ and define 
$$
f_n=\begin{cases} \psi_n, &  \quad  \text{if $n$ odd;} \\  
\psi_n/2Y, & \quad  \text{if $n$ even.}  \end{cases}
$$ 
%According to~\cite[Exercise~3.7(a)]{Silv}, 
By  Lemma~\ref{lem:polyz} we have
$f_n \in \mathbb{Z}[X]$. 
 By Lemma~\ref{lem:torsion1}, the vanishing of $f_n(X)$ for $n$ odd or of $2Yf_n(X)$ for $n$ even characterises the kernal $[n]$ of $E_1$. Define $g_n \in \mathbb{Z}[X]$ in a similar fashion for the curve $E_2$, so that the vanishing of $g_n(X)$ for $n$ odd or of $2Yg_n(X)$ for $n$ even characterises the kernel $[n]$ of $E_2$.

  From Lemma~\ref{lem:div poly}, one has  
\begin{align*}
& \deg f_n(X), \deg g_n(X)  \le n^{2+o(1)},   \quad  \text{if $n$ is odd},  \\  & \deg (Y f_n(X)), \deg (Y g_n(X)) \le n^{2+o(1)},\quad \text{if  $n$ is even},
\end{align*} 
and 
$$
h(f_n), h(g_n) \leq n^{2+o(1)}.
$$ 
We can also see that
 \begin{align*}
 &  \deg \prod_{l=C+1}^L f_l(X) = \sum_{l=C+1}^L \deg f_l(X) \le L^{3+o(1)},\\
  & \deg   \prod_{l=C+1}^L g_l(X)  = \sum_{l=C+1}^L \deg g_l(X) \le L^{3+o(1)}. 
 \end{align*}   
Furthermore, by Lemma~\ref{lem3.1},
 \begin{align*}
  & h\left(\prod_{l=C+1}^L f_l(X) \right) \leq\sum_{l=C+1}^L  h(f_l(X)) +\log{2}\sum_{l=C+1}^L  \deg f_l(X)  \le  L^{3+o(1)},\\
  & h\left(\prod_{l=C+1}^L g_l(X) \right) \leq\sum_{l=C+1}^L  h(g_l(X)) +\log{2}\sum_{l=C+1}^L  \deg g_l(X)  \le  L^{3+o(1)}.
 \end{align*} 
%  &\leq c_2L^{3+o(1)} + c_1L^3 \log 2 \leq c_3 L^{3+o(1)},
% \end{align*}  

% and analogously for $h\left(\prod_{l=1}^L \varphi_l(X) \right)$ for a common constant $c_3$ not depending on $L$.
 Let 
 $$
 \fR = \left| \Res\left(\prod_{l=C+1}^L f_l(X),\prod_{l=C+1}^L g_l(X)\right)\right |.
 $$ 
 We see from the choice of $C$ that $\fR \ne 0$.  If this were false then for some $X_0\in \ov \Q$ and $\ell,k$ satisfying 
$$C+1\le \ell,k\le L,$$ we have 
$$f_{\ell}(X_0)=g_{k}(X_0)=0.$$
By construction of $f_{\ell},g_k$, there exists  
$$(P_1,P_2)\in E_{1}(\bar{\mathbb{Q}})^{\text{tors}}\times E_{2}(\bar{\mathbb{Q}})^{\text{tors}},$$ with 
$$x(P_1)=x(P_2)=X_0,$$
and $P_1,P_2$ have orders $\ell,k$ respectivley. Since $\ell,k\ge C+1$, this contradicts our choice of $C$  and thus $\fR \ge 1$.

 Applying Lemma~\ref{lemres} with 
 $$
 A(X)=\prod_{l=C+1}^L f_l(X) \mand B(X)=\prod_{l=C+1}^L g_l(X)
 $$ 
 gives
 \begin{equation}
\label{eq:resboundelliptic}
 \log\fR  \le L^{6+o(1)}.
 \end{equation}
For integer $n$ let $\omega(n)$ count the number of distinct prime divisors  of $n$.
 Combining the classic estimate 
$$\omega(n)\ll \frac{\log{n}}{\log\log(n+2)},
$$
(which follows from the trivial inequality $\omega(n)! \le n$ and the Stirling formula)
with~\eqref{eq:resboundelliptic}, we see that the number $E$ of exceptional primes $p$ satisfying
\begin{equation}
\label{eq:presdiv}
p \mid \Res\left(\prod_{l=C+1}^L f_l(X),\prod_{l=C+1}^L g_l(X)\right),
\end{equation}
is at most 
 $$
E \ll  \frac{\log (\fR+1)}{\log  \log \(\fR +2\) } \le L^{6+o(1)}.
$$
%= O\left(\dfrac{L^{6+o(1)}}{\log L}\right)$$
%primes $p$ satisfying
%\begin{equation}
%\label{eq:presdiv}
%p \mid \Res\left(\prod_{l=C+1}^L \psi_l(X),\prod_{l=C+1}^L \varphi_l(X)\right).
%\end{equation}
 Recalling~\eqref{eq:Ldef1}, we see that 
  there are at most 
 $$
 L^{6+o(1)} \le z^{1-6\varepsilon +o(1)} = o(z/\log z)
 $$ primes $p \leq z$ satisfying~\eqref{eq:presdiv}.

By construction, for all the other remaining primes $p$ not dividing the conductors of $E_1$ and $E_2$, 
   the reduced elliptic curves $E_{1,p}$ and $E_{2,p}$ do not have any points $(\bar{P}_1,\bar{P}_2) \in E_{1,p}\times E_{1,p}$ with $x(\bar{P}_1)=x(\bar{P}_2)$ and $C+1 \leq \max \{ \text{ord } P_1, \text{ord } P_2\} \leq L$.
Thus, for such primes, for every pair of points $(\bar{P}_1,\bar{P}_2) \in E_{1,p}\times E_{1,p}$ with $x(\bar{P}_1)=x(\bar{P}_2)$ and 
$$\max \{ \text{ord } P_1, \text{ord } P_2\}>C,
$$ 
we have 
$$
\max \{ \text{ord } P_1, \text{ord } P_2\} \geq L= z^{1/6-\varepsilon}.
$$ Finally, there are at most $C^2C^2=C^4$ points in $E_1\times E_2$ with components of order at most $C$. Taking $C_0=C^4$,  and taking into account that $\varepsilon$ is arbitrary,  we conclude the proof.

\subsection{Proof of Theorem~\ref{th2.9}}
Throughout the proof, all constants depend only on  $f_1, \ldots, f_n$.
 For simplicity, we suppose that the heights of $f_1, \ldots, f_n$ are at most $h$. 
% 
% Without loss of generality we can assume that the function $\varepsilon(z)z^{1/n^2+2n-4}$ 
%is monotonically increasing and tends to infinity as $z  \to \infty$.
% 

  We see from Lemma~\ref{lem:multidep} that there is  
 %%an absolute integer 
 a constant $B$, which depends only on $f_1, \ldots, f_n$,  such that for any $x \in \overline{\mathbb{Q}}$ satisfying 
$$
f_1^{k_1}(x)\ldots f_n^{k_n}(x)=f_1^{\ell_1}(x)\ldots f_n^{\ell_n}(x)=1 
$$
 for some linearly independent  integer vectors 
$(k_1,\ldots ,k_n), (\ell_1,\ldots ,\ell_n)$ implies $ \ord (f_1(x),\ldots ,f_n(x)) \leq B$. 

We choose some large real number $z$ and define 
\begin{equation}
\label{eq:Ldef2}
L =\varepsilon(z)z^{1/(2n+2)}  
\end{equation} 
 Without loss of generality we can assume that  $\varepsilon(u)u^{1/(2n+2)}$ 
is a monotonically increasing  function of $u$ and so $L \ge \varepsilon(p)p^{1/(2n+2)}$ for any $p\le z$.
We also assume that $z$ is large enough and fixed so that 
$$
L  >B.
$$

Also for  a vector $\vec{k} = (k_1,\ldots ,k_n) \in \mathbb{Z}^n \setminus \{ \vec{0} \}$  we define
 \begin{equation}  
\label{eq:height-def}
\|\vec{k}\|_\infty = \max_{ i =1, \ldots, n} |k_i|.
 \end{equation}

%We aim to use Lemma~\ref{lem3.5}[BMZ] with $C= \mathbb{P}^1$.
Suppose each $f_i$ is of the form  
$$f_i=P_i/Q_i,\qquad i =1, \ldots, n,
$$ 
with coprime $P_i,Q_i \in \mathbb{Z}[X]$. 

   Let $\vec{k} = (k_1,\ldots ,k_n) \in \mathbb{Z}^n \setminus \{ \vec{0} \}$ satisfy
 \begin{equation}  
\label{eq:k-height}
\|\vec{k}\|_\infty \leq L.
 \end{equation}
We partition the multiset of components of $\vec{k}$ as  follows
 \begin{align*} 
& \{k_{u_1},\ldots , k_{u_r} \} = \{k_1,\ldots ,k_n \} \cap \mathbb{Z}_{\geq 0} , \\  
&  \{k_{v_1},\ldots , k_{v_s} \} = \{k_1,\ldots ,k_n \} \cap \mathbb{Z}_{<0} , 
 \end{align*} 
 and define
  \begin{equation} 
   \begin{split} 
\label{eq:Gk}
 G_\vec{k}(X) =    \prod_{i=1}^r   \prod_{j=1}^s  &P_{u_i}^{k_{u_i}}(X)Q_{v_j}^{-k_{v_j}}(X) \\
 & - 
 \prod_{i=1}^r \prod_{j=1}^s Q_{u_i}^{k_{u_i}}(X)   P_{v_j}^{-k_{v_j}}(X) \in \Z[X].
  \end{split}
  \end{equation}

 Using Lemmas~\ref{lem3.1}, and~\cite[Equation~(3.1)]{DOSS}, 
   we see that for $\vec{k}$ satisfying~\eqref{eq:k-height} we have 
  \begin{equation}\label{eq:deh h Gk}
  \deg  G_\vec{k}    \ll  L \mand   h\(  G_\vec{k} \)  \ll L.
  \end{equation} 
  
%\commB{Removed part of the argument using results from geometry of numbers and redefined $\mathscr{M}$ to consist of all points}
% The grid $$\mathscr{L}=[-L,L]^{n}\cap \Z^{n},$$
 % is contained in the ball $\cB(2L)^n \subset \mathbb{R}^n$ of radius $2L$. Therefore, by Lemma~\ref{lemlat} there exists some collection of two dimensional subspaces $S_1,\ldots,S_J$ satisfying 
%\begin{equation}
%\label{eq:Sproperties}
%\sL\subseteq \bigcup_{j=1}^{J}S_j \mand J=O\left(L^{n(n-2)/(n-1)}\right).
%\end{equation}  
% Note we may assume that each $S_j$ contains two linearly independent points $\vec{k},\vec{l} \in \sL$. 
% Indeed, otherwise there exists some $\vec{l} \in \sL$ such that 
%$$S_j\cap \sL\subseteq \{ z \vec{l}:~ z\in \R \} \cap \sL,$$
%and hence for any $\vec{k} \in \sL$ 
%$$S_j \cap \sL \subseteq \{ z_1 \vec{l} +z_2 \vec{k}:~ z_1,z_2\in \R \} \cap \sL,$$
%which allows us to replace $S_j$ a append $\cS_j$ with one of the unit vectors and just embed it in  some other two dimensional subspace with two linearly independent points in $\sL$ while still retaining the properties~\eqref{eq:Sproperties}. \commI{Reworded} 

%For each $S_j$,  we let $\vec{k}_j, \vec{l}_j\in \sL$ be a basis of the lattice $S_j\cap \Z^n$  and consider its associated basis $\vec{k}^*_j, \vec{l}^*_j$  given by Lemma \ref{lem:HB}, whose coordinates have absolute values $O(L)$. We enlarge $\sL$ with the vectors of such basis, and define
%$$
%\mathscr{M}=\bigcup_{1\le j \le J}\{ (\vec{k}_j,\vec{l}_j),(\vec{k}^*_j, \vec{l}^*_j) \},
%$$ 
%which 
%by~\eqref{eq:Sproperties}  has cardinality 
Define the set 
$$
\mathscr{M}=\{ (\vec{k},\vec{l})\in \Z^{n}\times \Z^{n} \ : \ \|\vec{k}\|_{\infty}, \|\vec{l}\|_{\infty}\le L \ \},
$$
so that
\begin{equation}
\label{eq:L0bound}
|\mathscr{M}|=O\left(L^{2n}\right).
\end{equation}
We recall the definition~\eqref{eq:height-def} and note that each such pair 
$$(\vec{k},\vec{l})=(k_1,\ldots ,k_n,\ell_1,\ldots ,\ell_n)\in \sM
$$ 
with 
$$
\|\vec{k}\|_\infty, \|\vec{l}\|_\infty  \ll L
$$  
leads to two polynomials whose set of common zeros is the set of elements in $x \in \overline{\mathbb{Q}}$ such that the multiplicative dependence relations 
\begin{equation}
\label{eq:fk=fl=1}
   \prod_{i=1}^n  f_i^{k_i}(x)=   \prod_{i=1}^n f_i^{\ell_i}(x)=1
\end{equation}
are satisfied.
% According to~\eqref{eq:deh h Gk}, the two polynomials defining the equations
%in~\eqref{eq:two prod} have degrees bounded  and height bounded by $O(L)$.   \commI{Removed $n$, $h$ from $O(*)$} 
%\hrule 
%\commI{Stopped here} 

 We recall the definition~\eqref{eq:Gk} and 
 for each $ \vec{k}, \vec{l} \in \mathscr{M}$ consider the following union of zero sets ranging over $\mathscr{M}$
$$
\bigcup_{(\vec{k}, \vec{l}) \in\sM} \{x \in \overline{\mathbb{Q}} :~G_\vec{k}(x)= G_\vec{l}(x)=0 \}. 
$$
 This set is finite  due to Lemma~\ref{lem:multidep} as it   differs from the union of zeros of the system~\eqref{eq:fk=fl=1}
  only by a subset of the set of zeros or poles of $f_1,\ldots, f_n$.  
More precisely,  in order to avoid counting  the zeros or poles coming from powers of the $P_i,Q_i$,  we define
 \begin{align*}
& \widetilde G_\vec{k}=G_\vec{k}/\gcd\(G_\vec{k}, \prod_{i=1}^n \(P_i Q_i\)^{O(L)}\), \\
& \widetilde G_\vec{l}=G_\vec{l}/\gcd\(G_\vec{l},  \prod_{i=1}^n \(P_i Q_i\)^{O(L)}\),
 \end{align*}
and consider 
$$
\bigcup_{(\vec{k}, \vec{l}) \in \mathscr{M}} \{x \in \overline{\mathbb{Q}} :~\widetilde G_\vec{k}(x)=\widetilde G_\vec{l}(x)=0 \}. 
$$
%%We now define $\widetilde G_\vec{k}\in \Z[X]$ and $\widetilde G_\vec{l}\in \Z[X]$  that the largest divisors of $G_\vec{k}$ and $G_\vec{l}$, 
%%respectively, 
%%which are relatively prime to $p_1, q_1, \ldots, p_n, q_n$ and which have relatively prime coefficients. 
%%\commI{There were lots of comments here. I defined tilded $G_\vec{k}$ and $G_\vec{l}$ in a different way and
%%insisted on "normalising them} 
As in~\eqref{eq:deh h Gk}, we see that we trivially have
\begin{equation}
\label{eq:degHG} 
 \deg \widetilde G_\vec{l}, \deg \widetilde G_\vec{k}   \ll  L 
\end{equation}
and also   by Lemma~\ref{lem3.1} 
\begin{equation}
\label{eq:heightHG}
h\left(\widetilde G_\vec{k}\right), h\left(\widetilde G_\vec{l}\right)\le h(G_\vec{k})+O\left(nLd \right)\ll L.
\end{equation} 

 Applying Lemma~\ref{lemres} with $A=\widetilde G_\vec{k}$ and $B=\widetilde G_\vec{l}$ and using~\eqref{eq:degHG} and ~\eqref{eq:heightHG} gives
$$
\log |\Res(\widetilde G_\vec{k},\widetilde G_\vec{l})| \ll L^2.
$$  
Arguing as in the proof of Theorem~\ref{th2.3}, there are at most
$$
O\left(\dfrac{\log |\Res(\widetilde G_\vec{k},\widetilde G_\vec{l})|+1}{\log \left(\log |\Res(\widetilde G_\vec{k},\widetilde G_\vec{l})| +2 \right) } \right) =  O\left( \dfrac{L^2}{\log L} \right)
$$
 primes $p \mid \Res(\widetilde G_\vec{k},\widetilde G_\vec{l})$. By~\eqref{eq:L0bound} there are at most
$ O(L^{2(n+1)}/\log L)$  primes dividing $\Res(G_\vec{k},G_\vec{l})$ for some $(\vec{k}, \vec{l}) \in \sM$. 

We also need to exclude primes $p$ such that the polynomials  
$$
 \widetilde G_\vec{k} \widetilde G_\vec{l} \mand \prod_{i=1}^n P_i Q_i,
$$ 
have a common zero over $\ov \F_p$ for some $(\vec{k},\vec{l})\in \sM$. Consider the resultants
 $$
 \fR_{\vec{k},\vec{l}} = \left| \Res\( \widetilde G_\vec{k} \widetilde G_\vec{l},   \prod_{i=1}^n  P_iQ_i \) \right |.
 $$
 By our construction $\fR_{\vec{k},\vec{l}} \ne 0$ for each $(\vec{k},\vec{l})\in \sM$. 
Arguing as above, we show 
$$\fR=\prod_{(\vec{k},\vec{l})\in \sM}\fR_{\vec{k},\vec{l}},$$ has a small number of prime divisors. We have
$$
\deg\prod_{i=1}^n  P_iQ_i \mand h\left(\prod_{i=1}^n  P_iQ_i\right) \ll 1,
$$
with implied constant depending only on $n,d,h$. Using~\eqref{eq:heightHG},~\eqref{eq:degHG}  we derive
$$
 \deg \left(\widetilde G_\vec{k} \widetilde G_\vec{l}    \right) \ll L
\mand 
 h \left(  \widetilde G_\vec{k} \widetilde G_\vec{l}    \right)  \ll L.
$$
  Hence by Lemma~\ref{lemres}
$$
\log \fR_{\vec{k},\vec{l}} \ll L,
$$
which by~\eqref{eq:L0bound} implies 
$$
\log \fR\le \sum_{(\vec{k},\vec{l})\in \sM}\log \fR_{\vec{k},\vec{l}}\ll L^{2n+1}, 
$$
and hence 
$$
\frac{\log (\fR+1)}{\log  \log (\fR+2)} \ll  \frac{L^{2n+1}}{\log L}.
$$
This implies that there are at most $O(L^{2n+2}/\log L)$ primes satisfying 
$$
p\mid \Res(\widetilde G_\vec{k},\widetilde G_\vec{l}) \quad  \text{for some $(\vec{k},\vec{l}) \in \mathscr{L}$} \qquad \text{or} \qquad   p\mid \fR.
$$

By~\eqref{eq:Ldef2}, there are at most $o(z/\log z)$ primes $p\leq z$ satisfying the above properties.
By construction, for all the other remaining primes $p$, there is no $x \in\ov \F_p$ satisfying 
 \begin{equation}\label{eq:fin cong}
 f_1^{k_1}(x)\ldots f_n^{k_n}(x)\equiv f_1^{\ell_1}(x)\ldots f_n^{\ell_n}(x) \equiv 1 \pmod  p
 \end{equation} 
for some linearly independent  integer vectors 
$(k_1,\ldots ,k_n), (\ell_1,\ldots ,\ell_n)\in \Z^n$  with $ B+1\leq \ord (f_1(x),\ldots ,f_n(x))  \leq L$.
In this case, we have 
$$
\ord (f_1(x),\ldots ,f_n(x)) \geq L= \varepsilon(z)z^{1/(2n+2)}.
$$ Finally,  for such primes,  there are at most $C_0 = dnB(2B)^n$ values of 
$x \in\ov \F_p$ satisfying a congruence~\eqref{eq:fin cong}  
 for some linearly independent vectors $(k_1,\ldots ,k_n), (\ell_1,\ldots ,\ell_n)\in \Z^n\cap [-B,B]^n$. 
 This implies there are at most $C_0$ values of $x \in\ov \F_p$  satisfying  $ \ord (f_1(x),\ldots ,f_n(x))  \leq L$.

\appendix
\section{Properties of division polynomials of Legendre curves} 
\label{app:div poly} 

Here we reproduce the proof of several results from~\cite[Chapter~4]{Ho} which
together yield  Lemma~\ref{lem:div poly}. 

\begin{lemma} \label{lem:polyz}
We have	$\psi_n \in \Z[\lambda, X, Y]$ and for an even $n=2k$ we also have $\psi_{2k} Y^{-1} \in \Z[\lambda, X, Y]$. 
\end{lemma}  

\begin{proof} The  obviously holds for $n \le 4$.  By induction, it also holds for $n = 2k+1$, $k = 2,3,\ldots$.

Let
$$
\psi_4 = (2Y)f \quad\text{where }f \in \Z[\lambda,X]. 
$$
Then
	\begin{align*}
	\psi_6 &= (2Y)^{-1}\psi_3(\psi_5\psi^2_2 - \psi_1\psi^2_4) \\
	&= (2Y)^{-1}(\psi_5(2Y)^2 - (2Y)^2f^2) \\
	&= (2Y)(\psi_5-f^2).
	\end{align*}
	We can therefore suppose that $\psi_{2k}$ has a factor of $2Y$.  Since
	\[ \psi_{2k+2} = (2Y)^{-1}\psi_{k+1}(\psi_{k+3}\psi^2_k - \psi_{k-1}\psi^2_{k+2}), \]
	we have 2 possible cases :	
	\begin{itemize}
		\item if $k$ is odd, then we get a factor of $2Y$ from $\psi_{k+1}, \psi_{k+3}$ and $\psi_{k-1}$, which, after cancellation, leaves $2Y$ as a factor of $\psi_{2k+2}$
		\item if $k$ is even, then we get a factor of $(2Y)^2$ from $\psi^2_k$ and $\psi^2_{k+2}$, which, after cancellation, also leaves $2Y$ as a factor of $\psi_{2k+2}$.
	\end{itemize}
	
	This proves that $\psi_{2k} = (2Y)g$ where $g \in \Z[\lambda, X, Y]$.  Hence $\psi_n \in \Z[\lambda, X, Y]$.
\end{proof}
\begin{lemma} \label{lem:polydeg} We have, 
	$\deg \psi_n \le n^{2+o(1)}.$
\end{lemma}  

\begin{proof}
This bound is equivalent to the statement that for each $\varepsilon>0$, there exists some constant $c(\varepsilon)$ such that for each $n=1,2,\ldots$ we have 
\begin{equation}
\label{eq:toprovedeg}
\deg \psi_n\le c(\varepsilon)n^{2+\varepsilon}.
\end{equation}
We prove~\eqref{eq:toprovedeg} by induction. Fix some $\varepsilon>0$ and choose $n_0$ large enough so that 
for 
$$
 k\ge \frac{n_0-1}{2}
 $$ 
 we have 
\begin{equation}
\label{eq:nchoice11}
\frac{(k+2)^{2+\varepsilon}}{2^{\varepsilon}(k+1/2)^{2+\varepsilon}},\  \frac{(k+1)^{2+\varepsilon}}{2^{\varepsilon}k^{2+\varepsilon}}<1.
\end{equation} 
Define $c(\varepsilon)$ by 
$$
c(\varepsilon)=\max_{n\le n_0}\frac{\deg \psi_n}{n^{2+\varepsilon}}.
$$
With this choice of $c(\varepsilon)$ the inequality~\eqref{eq:toprovedeg} trivially holds for $n\le n_0$ which forms the basis of our induction. Suppose~\eqref{eq:toprovedeg} is true for all integers $n<m$ for some $m>n_0$. Consider $m$ even or odd separatley. If $m=2k$ is even, then  
\begin{align*}
\deg \psi_{2k} & \le   \max\{ \deg \psi_k+\deg\psi_{k+1}+2\deg\psi_{k-1},\\
& \qquad \qquad \qquad \deg \psi_k+\deg \psi_{k-2}+2\deg \psi_{k+1}\}.
\end{align*}
By our induction hypothesis and~\eqref{eq:nchoice11}
$$
\deg \psi_{2k}\le c(\varepsilon)(2k)^{2+\varepsilon}\frac{(k+1)^{2+\varepsilon}}{2^{\varepsilon}k^{2+\varepsilon}}<c(\varepsilon)(2k)^{2+\varepsilon}.
$$
If $m=2k+1$ is odd, then by our induction hypothesis and~\eqref{eq:nchoice11}
\begin{align*}
\deg \psi_{2k+1} &\le  \max\{ \deg \psi_{k+2}+3\deg\psi_{k},\deg \psi_{k-1}+3\deg \psi_{k+1}\} \\ 
&\le c(\varepsilon)(2k+1)^{2+\varepsilon}\frac{(k+2)^{2+\varepsilon}}{2^{\varepsilon}(k+1/2)^{2+\varepsilon}}<c(\varepsilon)(2k+1)^{2+\varepsilon},
\end{align*}
which implies~\eqref{eq:toprovedeg} and concludes the proof. 
\end{proof}

\begin{lemma} \label{lem:polyht}
	$h(\psi_n) \le n^{2+o(1)}$ 
\end{lemma}

\begin{proof}
	The bound clearly holds for $n \le 4$ as $h(\psi_n) \le n \le n^2$. 
	
	This bound is also equivalent to the statement that, for any $\varepsilon > 0$, there exists some constant $c(\varepsilon)$ such that for every $j = 1,2,\ldots$, we have
	\begin{equation} \label{hbound}
	h(\psi_j) \le c(\varepsilon) j^{2 + \varepsilon}. 
	\end{equation}
	We fix some $n > 4$ and assume for induction that~\eqref{hbound} holds for $j < n$. 
	
	For $n = 2k+1$, from Lemmas~\ref{lem3.1} and~\ref{lem3.11} and the bound on the degree of division polynomials 
given in Lemma~\ref{lem:polydeg}, we have
	\[ h(\psi_{2k+1}) \le \max\{ h(\psi_{k+2})+3h(\psi_k), h(\psi_{k-1})+3h(\psi_{k+1}) \} + c_0 k^2 \]
	with some constant $c_0$.  By the induction assumption, we can estimate all heights on the right hand side by $c(\varepsilon)(k+2)^{2+\varepsilon}$ and obtain
	\begin{equation} \label{ineq:1}
	\begin{split}
	h(\psi_{2k+1}) & \le 4c(\varepsilon)(k+2)^{2+\varepsilon} + c_0 k^2 \\
	& = 4c(\varepsilon)k^{2+\varepsilon}\left((1 + 2/k)^{2+\varepsilon} + c_0 k^{-\varepsilon} \right).
	\end{split}
	\end{equation}
	By increasing the value of $c(\varepsilon)$, we can assume that $k$ is large enough such that
	$$ (1 + 2/k)^{2+\varepsilon} + c_0k^{-\varepsilon} \le 4^{\varepsilon/2}.
	$$
	By substituting this in the previous inequality, we get
	\[ h(\psi_{2k+1}) \le 4^{1+\varepsilon/2}c(\varepsilon)k^{2+\varepsilon} = c(\varepsilon)(2k)^{2+\varepsilon} < c(\varepsilon)(2k+1)^{2+\varepsilon}. \]
	
	For $n = 2k$, by the same reasoning, we obtain the same inequality as in~\eqref{ineq:1} and reach the desired inequality
	\[ h(\psi_{2k}) \le c(\varepsilon)(2k)^{2+\varepsilon}. \]
	
	Hence, $h(\psi_n) \le n^{2+o(1)}$.
\end{proof}

\section*{Acknowledgement}

We also would like to thank the authors of~\cite{BCMOS} for sending us a preliminary version of their 
work and many very useful comments. 

We are also grateful to Joshua Ho for his permission to reproduce some parts of  his master's thesis~\cite{Ho}.

During the preparation of this work, B.K. was  supported  by Australian Research Council Grant~DP160100932 and  Academy of Finland Grant~319180, 
J.G.  
by Australian Research Council Grant~~DP180100201 and I.S.  
by Australian Research Council Grants DP170100786 and DP180100201.

\end{document}